\documentclass{amsart}
\usepackage{amssymb}
\usepackage{amsfonts,color}
\usepackage{graphicx}
\usepackage{hyperref}
\newtheorem{theorem}{Theorem}[section]
\theoremstyle{plain}

\newtheorem{corollary}{Corollary}[section]

\newtheorem{definition}{Definition}[section]
\newtheorem{example}{Example}[section]

\newtheorem{lemma}{Lemma}[section]

\newtheorem{proposition}{Proposition}[section]
\newtheorem{remark}{Remark}[section]

\numberwithin{equation}{section}
\input{tcilatex}

\begin{document}
\title[]{On invertible nonnegative Hamiltonian operator matrices}
\author{Guohai Jin, Guolin Hou, Alatancang Chen, Deyu Wu}
\address{School of Mathematical Sciences\\ Inner Mongolia University\\ Hohhot\\ 010021\\ China}
\email{ghjin2006@gmail.com (G. Jin), smshgl@imu.edu.cn (G. Hou)}
\email{alatanca@imu.edu.cn (Alatancang), wudeyu2585@163.com (D. Wu)}
\thanks{Corresponding author: Guolin Hou}
\thanks{This paper is in final form and no version of it will be submitted
for publication elsewhere.}
\thanks{Fax:+86-471-4991650}

\begin{abstract}
Some new characterizations of nonnegative Hamiltonian operator matrices are given.
Several necessary and sufficient conditions for an unbounded nonnegative Hamiltonian operator to be invertible are obtained,
so that the main results in the previously published papers are corollaries of the new theorems.
Most of all we want to stress the method of proof.
It is based on the connections between Pauli operator matrices and nonnegative Hamiltonian matrices.
\end{abstract}
\maketitle

\section{Introduction}

A \emph{Hamiltonian operator matrix} is a block operator matrix
\begin{align}\label{eq1.1}
H=\left(
    \begin{array}{cc}
      A & B \\
      C & -A^* \\
    \end{array}
  \right)
\end{align}
acting on the product space $X\times X$ of some complex Hilbert space $X$ with closed densely defined operators $A, B, C$
such that $B, C$ are self-adjoint and $H$ is densely defined \cite{AJW}.
If, in addition, $B$ and $C$ are nonnegative, then $H$ is said to be a nonnegative Hamiltonian operator matrix \cite{Ku1}.

There are a number of very interesting ways that Hamiltonian operator matrices can arise.
We mention a few.
First, many linear boundary value problems in mathematical physics can be written as the Hamiltonian equation $\dot{u}=Hu+f$ where
$H$ is a Hamiltonian operator matrix,
so that the solvability of the original boundary value problem is reduced to spectral properties of the operator $H$,
see e.g. \cite{Atk, Krall2002} for ordinary differential equations,
\cite{CM, Olver, ZAZ} for partial differential equations, and \cite{LX,YZL,Zhong,ZZ} for applications of Hamiltonian operators in elasticity.
Second, Hamiltonian operator matrices also arise in theory of optimal control.
It is well known that the solutions $U$ of the Riccati equation
$$A^*U+UA+UBU-C=0$$
are in one-to-one correspondence with graph subspaces that are invariant under the operator matrix $H$ given by \eqref{eq1.1},
where $A, B, C$ are unbounded linear operators and $B, C$ are nonnegative,
see e.g. \cite{TW, Wyss2011} and the references therein.
There have been a lot of papers on spectral properties of Hamiltonian operator matrices,
see e.g. \cite{AHF,AD1,AKK,KZ,Ku1,KM,LRW,QC1,QC2,SS,WA}.
There are many papers \cite{Deni, Ku1, WA} devoted to the invertibility of a nonnegative Hamiltonian operator matrix
since it is sometimes important in the investigation of Hamiltonian equations \cite{Ku1} and, moreover,
for an invertible Hamiltonian operator matrix $H$ we have $JH=(JH)^*$
where $J=\left(
\begin{array}{cc}
0 & I \\
-I & 0 \\
\end{array}
\right)$
is the unit symplectic operator matrix \cite[p. 11]{YZL},
so that the spectral theorems hold \cite{AHF}.
Invertible Hamiltonian operator matrices also play an important role in spectral theory of periodic waves for infinite-dimensional Hamiltonian systems,
see e.g. \cite{HK} and references therein.

The purpose of this paper is to establish some necessary and sufficient conditions
for a nonnegative Hamiltonian operator matrix to be invertible.
Let us list and comment on some main results that are previously published.
Kurina \cite{Ku1} obtained the following fundamental theorem.
\begin{theorem}(\cite{Ku1})\label{th1.1}
Let $H$ be a nonnegative Hamiltonian operator matrix given by \eqref{eq1.1} with bounded off-diagonal entries $B$ and $C$
such that $0\in\rho(A)\cup(\rho(B)\cap\rho(C))$.
Then $0\in\rho(H)$.
\end{theorem}
Note that a Hamiltonian operator matrix $H$ with bounded off-diagonal entries satisfies $JH=(JH)^*$.

Denisov \cite{Deni} extended the case $0\in\rho(B)\cap\rho(C)$ of Kurina's
to nonnegative Hamiltonian operator matrices with unbounded off-diagonal entries.
\begin{theorem}(\cite[Theorem 2]{Deni})\label{th1.2}
Let $H$ be a nonnegative Hamiltonian operator matrix given by \eqref{eq1.1} such that $0\in\rho(B)\cap\rho(C)$.
If $iH$ is maximal $\sigma_1$-dissipative where
$\sigma_1=\left(
            \begin{array}{cc}
              0 & I \\
              I & 0 \\
            \end{array}
          \right)$,
then $0\in\rho(H)$.
\end{theorem}
Note that ``$iH$ is maximal $\sigma_1$-dissipative" can be replaced by ``$JH=(JH)^*$" in the statement of Theorem \ref{th1.2},
see Proposition \ref{prop4.3} in Section \ref{sec4}.

Wu and Alatancang \cite{WA} extended the main result of Kurina's
to nonnegative Hamiltonian operator matrices with unbounded off-diagonal entries.
\begin{theorem}(\cite[Theorem 3.1, Proposition 3.3]{WA})\label{th1.3}
Let $H$ be a nonnegative Hamiltonian operator matrix given by \eqref{eq1.1} with $\mathcal{D}(H)=\mathcal{D}(A)\times\mathcal{D}(A^*)$
such that $0\in\rho(A)\cup(\rho(B)\cap\rho(C))$.
If  $JH=(JH)^*$, then $0\in\rho(H)$.
\end{theorem}
We point out that there is a gap in the proof of \cite[Theorem 3.1]{WA}
for the equality $JH=(JH)^*$,
so that Theorem $3.1$ or Proposition $3.3$ in \cite{WA} holds under an additional hypothesis that $JH=(JH)^*$.

Wu and Alatancang \cite{WA} also obtained the following theorem.
\begin{theorem}\label{th1.4}(\cite[Theorem 3.2]{WA})
Let $H$ be a nonnegative Hamiltonian operator given by \eqref{eq1.1} with $\mathcal{D}(H)=\mathcal{D}(C)\times\mathcal{D}(B)$
such that $0\in\rho(A)$.
If the operators $A(C-\lambda)^{-1}, A^*(B-\lambda)^{-1}$ are compact for some negative number $\lambda$,
then $0\in\rho(H)$.
\end{theorem}

We conclude that all the theorems listed above are corollaries of the new theorems in Section \ref{sec3}.
To do this, we shall give some new characterizations of nonnegative Hamiltonian operator matrices in Section \ref{sec4}.

\section{preliminaries}

Throughout the remainder of this paper $X, Y$ will denote complex Hilbert spaces.
\begin{definition}(\cite[p. 3]{GK})
Let $T:\mathcal{D}(T)\subset X\to X$ be a linear operator.
$T$ is said to be invertible if it has an everywhere defined bounded inverse.
\end{definition}
Obviously, an invertible linear operator is closed and, moreover,
if $H$ is closed, then $H$ is invertible if and only if $0\in\rho(H)$.

\begin{definition}(\cite[p. 28, p. 41]{Hal})
Let $T:\mathcal{D}(T)\subset X\to Y$ be a linear operator.
$T$ is said to be bounded from below if there exists a positive number $\delta$
such that $\|Tx\|\geq\delta\|x\|$ for all $x\in\mathcal{D}(T)$.
If $X=Y$, the approximate point spectrum of $T$ is defined as
$$\sigma_{app}(T):=\{\lambda\in\mathbb C~|~(T-\lambda)\mbox{~is not bounded from below}\}.$$
\end{definition}

The following two lemmas will paly a role in the proofs of some theorems in the next section.
\begin{lemma}\label{lem2.1}
Let $H=\left(
         \begin{array}{cc}
           A & B \\
           C & -A^* \\
         \end{array}
       \right)
$ be a nonnegative Hamiltonian operator matrix on $X\times X$.
Then $H$ is bounded from below if one of the following holds:
\begin{enumerate}
\item $0\in\rho(B)\cap\rho(C)$,
\item $\mathcal{D}(H)=\mathcal{D}(A)\times\mathcal{D}(A^*)$ and $0\in\rho(A)$.
\end{enumerate}
\end{lemma}
\begin{proof}
The first assertion follows from
${\rm Im}(\sigma_1(iH)x,x)=(Cx_1,x_1)+(Bx_2,x_2)$ for all $x=(x_1\ x_2)^t\in\mathcal{D}(H)$,
see the proof of \cite[Theorem 2]{Deni}.
For the second assertion see the proof of \cite[Theorem 3.1]{WA} or \cite[Lemma 4.5]{TW}.
\end{proof}

\begin{lemma}(\cite[Corollary 1]{HK1970})\label{lem2.2}
Let $T$ be a closed densely defined linear operator on a Hilbert space.
Suppose $S$ is a $T$-bounded operator such that $S^*$ is $T^*$-bounded, with both relative bounds $<1$.
Then $S+T$ is closed and $(S+T)^*=S^*+T^*$.
\end{lemma}

Finally, we give two definitions that are important to Theorem \ref{th3.1} in Section \ref{sec3}.
\begin{definition}(\cite[Section 2.2]{Azi})
Let $\mathcal{J}:X\to X$ be an unitary self-adjoint linear operator.
A linear operator $T$ is said to be $\mathcal{J}$-dissipative
if ${\rm Im}(\mathcal{J}Tx,x)\geq 0$ for all $x\in\mathcal{D}(T)$,
and to be maximal $\mathcal{J}$-dissipative if it is $\mathcal{J}$-dissipative and coincides with any
$\mathcal{J}$-dissipative extension of it.
\end{definition}

\begin{definition}
Let $\mathcal{J}:X\to X$ be a linear operator such that $\mathcal{J}$ or $i\mathcal{J}$ is unitary self-adjoint.
A densely defined linear operator $T$ is said to be $\mathcal{J}$-symmetric
if $\mathcal{J}T\subset(\mathcal{J}T)^*$,
and to be $\mathcal{J}$-self-adjoint if $\mathcal{J}T=(\mathcal{J}T)^*$.
\end{definition}
Obviously,
\begin{align*}
\mathcal{J}T\subset(\mathcal{J}T)^*\iff\mathcal{J}T\mathcal{J}\subset T^*\iff T\subset\mathcal{J}T^*\mathcal{J},\\
\mathcal{J}T=(\mathcal{J}T)^*\iff\mathcal{J}T\mathcal{J}=T^*\iff T=\mathcal{J}T^*\mathcal{J}.
\end{align*}

\section{main results}\label{sec3}

In this section $H=\left(
    \begin{array}{cc}
      A & B \\
      C & -A^* \\
    \end{array}
  \right)$
will denote a nonnegative Hamiltonian operator matrix on $X\times X$,
and moreover, $J:=\left(
                \begin{array}{cc}
                  0 & I \\
                  -I & 0 \\
                \end{array}
              \right), \sigma_1:=\left(
                                   \begin{array}{cc}
                                     0 & I \\
                                     I & 0 \\
                                   \end{array}
                                 \right)$.

First we give the following necessary and sufficient conditions for $H$ to be invertible;
the proof will be given in Section \ref{sec4}.
\begin{theorem}\label{th3.1}
The following statements are equivalent:
\begin{enumerate}
\item $H$ is invertible,
\item $\mathcal{R}(H)=X\times X$,
\item $H$ is $J$-self-adjoint and bounded from below,
\item $H$ is $J$-self-adjoint with closed range and
$$\mathcal{N}(A)\cap\mathcal{N}(C)=\mathcal{N}(B)\cap\mathcal{N}(A^*)=\{0\},$$
\item $iH$ is maximal $\sigma_1$-dissipative and bounded from below,
\item $iH$ is maximal $\sigma_1$-dissipative with closed range and
$$\mathcal{N}(A)\cap\mathcal{N}(C)=\mathcal{N}(B)\cap\mathcal{N}(A^*)=\{0\}.$$
\end{enumerate}
\end{theorem}

\begin{remark}
We see from Lemma \ref{lem2.1} that Theorem \ref{th1.1}, Theorem \ref{th1.2}, and Theorem \ref{th1.3} are corollaries of Theorem \ref{th3.1}.
\end{remark}

\begin{remark}
The statements (1), (2), (3) of Theorem \ref{th3.1} are equivalent for general Hamiltonian operators (not necessarily nonnegative).
\end{remark}

\begin{remark}
The assumption about nonnegativity of the Hamiltonian operator matrix
is essential for the statement (4), (5), (6),
see Proposition \ref{prop4.1}, Proposition \ref{prop4.5} in Section \ref{sec4} and Example \ref{exam3.1} below.
\end{remark}

\begin{remark}
By Proposition \ref{prop4.4} in Section \ref{sec4}, ``$H$ is $J$-self-adjoint" (``$iH$ is maximal $\sigma_1$-dissipative", respectively)
can be replaced by ``~$\mathcal{R}(H+\sigma_1)=X\times X$".
\end{remark}

\begin{corollary}\label{cor3.1}
Suppose $H$ is $J$-self-adjoint.
Then $i\mathbb R\subset\rho(H)$ if one of the following statements holds:
\begin{enumerate}
\item $0\in\rho(B)\cap\rho(C)$,
\item $\mathcal{D}(H)=\mathcal{D}(A)\times\mathcal{D}(A^*)$ and $i\mathbb R\subset\rho(A)$.
\end{enumerate}
\end{corollary}

\begin{proof}
For each $\lambda\in\mathbb R$,
$$H-i\lambda=\left(
               \begin{array}{cc}
                 A-i\lambda & B \\
                 C & -(A-i\lambda)^* \\
               \end{array}
             \right)$$
is a nonnegative Hamiltonian operator matrix with the property $J(H-i\lambda)=(J(H-i\lambda))^*$.
Thus, by Theorem \ref{th3.1} and Lemma \ref{lem2.1}, $i\lambda\in\rho(H)$.
\end{proof}
Since conditions for $H$ to be invertible given by its entry operators $A, B, C$ are much more useful than those given by $H$ itself,
the rest of this section is devoted to the former.
We shall consider the two cases that $\mathcal{D}(H)=\mathcal{D}(A)\times\mathcal{D}(A^*)$ and $\mathcal{D}(H)=\mathcal{D}(C)\times\mathcal{D}(B)$.

First we consider the case $\mathcal{D}(H)=\mathcal{D}(A)\times\mathcal{D}(A^*)$.
The following is an improvement of Theorem \ref{th1.1}.
\begin{theorem}\label{th3.2}
Assume both $B$ and $C$ are bounded.
Then $H$ is invertible if and only if
both $\left(
      \begin{array}{c}
        A \\
        C \\
      \end{array}
    \right)$ and
$\left(
      \begin{array}{c}
        B \\
        -A^* \\
      \end{array}
    \right)$ are bounded from below.
In particular, if $A, B$, and $C$ are $n\times n$ matrices such that $B$ and $C$ are nonnegative Hermitian matrices,
then $H$ is invertible if and only if ${\rm rank~}(A^*\ \ C)={\rm rank~}(B\ \ -A)=n$.
\end{theorem}

\begin{proof}
Obviously $H$ is $J$-self-adjoint.
By Theorem \ref{th3.1}, it is sufficient to show $H$ is bounded from below if and only if
both $\left(
      \begin{array}{c}
        A \\
        C \\
      \end{array}
    \right)$ and
$\left(
      \begin{array}{c}
        B \\
        -A^* \\
      \end{array}
    \right)$ are bounded from below.
The \emph{only if} part is trivial.
We start to prove the \emph{if} part.
If $H$ were not bounded from below,
then there exist $x^{(n)}=(x_1^{(n)}\ x_2^{(n)})^t\in\mathcal{D}(A)\times\mathcal{D}(A^*), \|x^{(n)}\|=1, n=1, 2, ...$, such that
$Hx^{(n)}\to 0$, i.e.,
\begin{align}\label{eq3.1}
Ax_1^{(n)}+Bx_2^{(n)}\to 0,\\
\nonumber Cx_1^{(n)}-A^*x_2^{(n)}\to 0,
\end{align}
and so
\begin{align*}
(Ax_1^{(n)},x_2^{(n)})+(Bx_2^{(n)},x_2^{(n)})\to 0,\\
\nonumber (x_1^{(n)},Cx_1^{(n)})-(x_1^{(n)},A^*x_2^{(n)})\to 0,
\end{align*}
consequently
$$(x_1^{(n)},Cx_1^{(n)})+(Bx_2^{(n)},x_2^{(n)})\to 0.$$
Since $B, C$ are nonnegative self-adjoint operator, we get
$$(x_1^{(n)},Cx_1^{(n)})\to 0, (Bx_2^{(n)},x_2^{(n)})\to 0,$$
i.e.,
\begin{equation}\label{eq3.2}
C^{\frac{1}{2}}x_1^{(n)}\to 0, B^{\frac{1}{2}}x_2^{(n)}\to 0.
\end{equation}
But $B^{\frac{1}{2}}, C^{\frac{1}{2}}$ are bounded since $B, C$ are bounded,
and therefore
$$Cx_1^{(n)}\to 0, Bx_2^{(n)}\to 0.$$
By \eqref{eq3.1},
$$Ax_1^{(n)}\to 0, Cx_1^{(n)}\to 0, Bx_2^{(n)}\to 0, -A^*x_2^{(n)}\to 0,$$
i.e.,
\begin{equation}\label{eq3.3}
\left(
    \begin{array}{c}
      A \\
      C \\
    \end{array}
  \right)x_1^{(n)}\to 0, \left(
    \begin{array}{c}
      B \\
      -A^* \\
    \end{array}
  \right)x_2^{(n)}\to 0.
\end{equation}
On the other hand, by the assumption there exists a positive number $\delta$ such that
\begin{align*}
\|\left(
    \begin{array}{c}
      A \\
      C \\
    \end{array}
  \right)x_1^{(n)}\|^2
+\|\left(
    \begin{array}{c}
      B \\
      -A^* \\
    \end{array}
  \right)x_2^{(n)}\|^2
\geq\delta(\|x_1^{(n)}\|^2+\|x_2^{(n)}\|^2)
=\delta,
\end{align*}
this contradicts \eqref{eq3.3}.
Hence $H$ is bounded from below.
\end{proof}

\begin{remark}\label{rmk3.5}
The assertion of Theorem \ref{th3.2} does not hold for nonnegative Hamiltonian operators with unbounded off-diagonal entries,
see Example \ref{exam3.2}.
\end{remark}
\begin{remark}\label{rmk3.6}
The assumption about nonnegativity of the Hamiltonian operator matrix
is essential, see Example \ref{exam3.1} below.
\end{remark}

We extend Theorem \ref{th1.1} to the following case.
\begin{proposition}\label{prop3.1}
Suppose that $C$ is $A$-bounded and $B$ is $A^*$-bounded with both relative bounds less than $1$,
and that $0\in\rho(A)\cup(\rho(B)\cap\rho(C))$.
Then $H$ is invertible.
\end{proposition}
\begin{proof}
It is sufficient to show $H$ is $J$-self-adjoint and bounded from below.
Since $C$ is $A$-bounded and $B$ is $A^*$-bounded with both relative bounds less than $1$,
one see easily from Lemma \ref{lem2.2} that $JH=(JH)^*$ (see also \cite{WA2}).
Moreover, $H$ is bounded from below by Lemma \ref{lem2.1}.
\end{proof}

We also extend Theorem \ref{th1.1} to another case.
\begin{theorem}\label{th3.3}
Suppose that $\mathcal{D}(H)=\mathcal{D}(A)\times\mathcal{D}(A^*)$, and that $\rho(A)\neq\emptyset$.
If $0\in\rho(A)\cup(\rho(B)\cap\rho(C))$, then the following statements are equivalent:
\begin{enumerate}
\item $H$ is invertible,
\item $A^*+\lambda+C(A-\lambda)^{-1}B=(A+\overline{\lambda}+B(A^*-\overline{\lambda})^{-1}C)^*$
for some (and hence for all) $\lambda\in\rho(A)$,
\item $A+\overline{\lambda}+B(A^*-\overline{\lambda})^{-1}C=(A^*+\lambda+C(A-\lambda)^{-1}B)^*$
for some (and hence for all) $\lambda\in\rho(A)$.
\end{enumerate}
\end{theorem}
\begin{proof}
By Lemma \ref{lem2.1}, we need to show that each one of the last two statements is equivalent to $JH=(JH)^*$,
but this follows from \cite[Theorem 3.1]{AJW}.
\end{proof}

There are three corollaries follow from Theorem \ref{th3.3} and the corresponding corollaries of \cite[Theorem 3.1]{AJW}.
\begin{corollary}
Suppose that $\mathcal{D}(H)=\mathcal{D}(A)\times\mathcal{D}(A^*)$, and that $\rho(A)\neq\emptyset$.
If $0\in\rho(A)\cup(\rho(B)\cap\rho(C))$,
then $H$ is invertible if one of the following holds:
\begin{enumerate}
\item $C$ is $A$-bounded with relative bound $0$,
\item $B$ is $A^*$-bounded with relative bound $0$.
\end{enumerate}
\end{corollary}

\begin{corollary}
Suppose that $\mathcal{D}(H)=\mathcal{D}(A)\times\mathcal{D}(A^*)$,
and that $A$ or $-A$ is maximal accretive.
If $0\in\rho(A)\cup(\rho(B)\cap\rho(C))$,
then $H$ is invertible if one of the following holds:
\begin{enumerate}
\item $C$ is $A$-bounded with relative bound $< 1$ and $B$ is $A^*$-bounded with relative bound $\leq 1$,
\item $C$ is $A$-bounded with relative bound $\leq 1$ and $B$ is $A^*$-bounded with relative bound $< 1$.
\end{enumerate}
\end{corollary}

\begin{corollary}
Suppose that $\mathcal{D}(H)=\mathcal{D}(A)\times\mathcal{D}(A^*)$, and that $A$ is self-adjoint.
If $0\in\rho(A)\cup(\rho(B)\cap\rho(C))$,
then $H$ is invertible if one of the following holds:
\begin{enumerate}
\item $C$ is $A$-bounded with relative bound $< 1$ and $B$ is $A^*$-bounded with relative bound $\leq 1$,
\item $C$ is $A$-bounded with relative bound $\leq 1$ and $B$ is $A^*$-bounded with relative bound $< 1$.
\end{enumerate}
\end{corollary}

In the following two theorems,
we connect invertibility of $H$ to Fredholmness of its entry operators.
\begin{theorem}\label{th3.4}
$H$ is invertible if
\begin{enumerate}
\item $\rho(A)\cap i\mathbb R\neq\emptyset$ and $(A-i\lambda)^{-1}$ is a compact operator for some $i\lambda\in\rho(A)\cap i\mathbb R$,
\item $C$ is $A$-compact and $B$ is $A^*$-compact,
\item $\mathcal{N}(A)\cap\mathcal{N}(C)=\mathcal{N}(B)\cap\mathcal{N}(A^*)=\{0\}$.
\end{enumerate}
\end{theorem}
\begin{proof}
It is sufficient to show $H$ is a $J$-self-adjoint operator with closed range.
We see from (2) that $C$ is $A$-bounded and $B$ is $A^*$-bounded with both relative bounds $0$,
so that $JH=(JH)^*$. Next we prove $H$ has a closed range.
Writing
\begin{equation*}
H=\left(
     \begin{array}{cc}
       A-i\lambda & 0 \\
       0 & -A^*-i\lambda \\
     \end{array}
   \right)+\left(
             \begin{array}{cc}
               i\lambda & B \\
               C & i\lambda \\
             \end{array}
           \right)
=:T_{\lambda}+S_{\lambda}.
\end{equation*}
It follows from the first two assumptions that $S_{\lambda}$ is $T_{\lambda}$-compact since
$$S_{\lambda}T_{\lambda}^{-1}=\left(
                                \begin{array}{cc}
                                  i\lambda(A-i\lambda)^{-1} & -B(A^*+i\lambda)^{-1} \\
                                  C(A-i\lambda)^{-1} & -i\lambda((A-i\lambda)^{-1})^* \\
                                \end{array}
                              \right)$$
is a compact operator, so that $H$ is Fredholm since $T_{\lambda}$ is Fredholm (see \cite[Theorem IV 5.26]{Ka1980}).
Hence $H$ has a closed range.
\end{proof}

\begin{theorem}\label{th3.5}
$H$ is invertible if
\begin{enumerate}
\item $A$ is Fredholm,
\item $C$ is $A$-compact and $B$ is $A^*$-compact,
\item $\mathcal{N}(A)\cap\mathcal{N}(C)=\mathcal{N}(B)\cap\mathcal{N}(A^*)=\{0\}$.
\end{enumerate}
\end{theorem}
\begin{proof}
We need to show $H$ has a closed range.
Writing
\begin{equation*}
H=\left(
     \begin{array}{cc}
       A & 0 \\
       0 & -A^* \\
     \end{array}
   \right)+\left(
             \begin{array}{cc}
               0 & B \\
               C & 0 \\
             \end{array}
           \right)
=:T+S.
\end{equation*}
Then, $T$ is Fredholm by (1)
and $S$ is $T$-compact by (2).
Thus, $H$ is Fredholm (see \cite[Theorem IV 5.26]{Ka1980}),
so that it has a closed range.
\end{proof}

Now we consider the case $\mathcal{D}(H)=\mathcal{D}(C)\times\mathcal{D}(B)$.
The following is an analogue of Theorem \ref{th3.3}.
\begin{theorem}\label{th3.6}
Suppose that $\mathcal{D}(H)=\mathcal{D}(C)\times\mathcal{D}(B)$, and that $0\in\rho(B)\cap\rho(C)$.
Then the following statements are equivalent:
\begin{enumerate}
\item $H$ is invertible,
\item $C+\lambda+A^*(B-\lambda)^{-1}A=(C+\overline{\lambda}+A^*(B-\overline{\lambda})^{-1}A)^*$
for some (and hence for all) $\lambda\in\rho(B)$,
\item $B+\lambda+A(C-\lambda)^{-1}A^*=(B+\overline{\lambda}+A(C-\overline{\lambda})^{-1}A^*)^*$
for some (and hence for all) $\lambda\in\rho(C)$.
\end{enumerate}
\end{theorem}
\begin{proof}
By Lemma \ref{lem2.1}, we need to show that each one of the last two statements is equivalent to $JH=(JH)^*$,
but this follows from \cite[Theorem 3.2]{AJW}.
\end{proof}
The following corollary follows from Theorem \ref{th3.6} and the corresponding corollary of \cite[Theorem 3.2]{AJW}.
\begin{corollary}\label{cor3.5}
Suppose that $\mathcal{D}(H)=\mathcal{D}(C)\times\mathcal{D}(B)$, and that $0\in\rho(B)\cap\rho(C)$.
Then $H$ is invertible if one of the following holds:
\begin{enumerate}
\item $A$ is $C$-bounded with relative bound $< 1$ and $A^*$ is $B$-bounded with relative bound $\leq 1$,
\item $A$ is $C$-bounded with relative bound $\leq 1$ and $A^*$ is $B$-bounded with relative bound $< 1$.
\end{enumerate}
\end{corollary}

\begin{remark}\label{rmk3.7}
``~$0\in\rho(B)\cap\rho(C)$" in the statement of Theorem \ref{th3.6} (Corollary \ref{cor3.5}, respectively) cannot be replaced by ``~$0\in\rho(A)\cup(\rho(B)\cap\rho(C))$",
see Example \ref{exam3.2}.
\end{remark}

The following analogues of Theorem \ref{th3.4} and Theorem \ref{th3.5} are proved in completely analogous ways.
\begin{theorem}\label{th3.7}
$H$ is invertible if
\begin{enumerate}
\item $(B-\lambda)^{-1}, (C-\lambda)^{-1}$ are compact operators for some $\lambda\in\rho(B)\cap\rho(C)$,
\item $A$ is $C$-compact and $A^*$ is $B$-compact,
\item $\mathcal{N}(A)\cap\mathcal{N}(C)=\mathcal{N}(B)\cap\mathcal{N}(A^*)=\{0\}$.
\end{enumerate}
\end{theorem}
\begin{remark}
Theorem \ref{th1.4} is a corollary of Theorem \ref{th3.7} since
the assumptions of the former imply those of the latter.
\end{remark}

\begin{theorem}
$H$ is invertible if
\begin{enumerate}
\item $B, C$ are Fredholm,
\item $A$ is $C$-compact and $A^*$ is $B$-compact,
\item $\mathcal{N}(A)\cap\mathcal{N}(C)=\mathcal{N}(B)\cap\mathcal{N}(A^*)=\{0\}$.
\end{enumerate}
\end{theorem}

The following simple example shows that the assumption about nonnegativity of the Hamiltonian operator matrix
is essential for the main results in this section, e.g., Theorem \ref{th3.1} and Theorem \ref{th3.2}.
\begin{example}\label{exam3.1}
Consider the $2\times 2$ Hamiltonian matrix
$$H=\left(
      \begin{array}{cc}
        1 & 1 \\
        -1 & -1 \\
      \end{array}
    \right)$$
on the Hilbert space $\mathbb C\times\mathbb C$.
Obviously, $H$ is $J$-self-adjoint with closed range and
$$\mathcal{N}(A)\cap\mathcal{N}(C)=\mathcal{N}(B)\cap\mathcal{N}(A^*)=\{0\}.$$
Moreover, $\left(
      \begin{array}{c}
        A \\
        C \\
      \end{array}
    \right),
\left(
      \begin{array}{c}
        B \\
        -A^* \\
      \end{array}
    \right)$ are bounded from below.
But $0\in\sigma(H)$ since $det H=0$.
\end{example}

The following example indicates that the assumptions of many theorems cannot be relaxed, see Remark \ref{rmk3.5} and Remark \ref{rmk3.7}.
See also \cite{Deni} for another example.
\begin{example}\label{exam3.2}
Let $C$ be an unbounded self-adjoint operator on $X$ such that $(C-\gamma)$ is nonnegative for some
positive number $\gamma$. Noting that $0\in\rho(C)$ and $0\in\sigma_c(C^{-1})$.
For the nonnegative Hamiltonian operator matrix
$$H:=\left(
      \begin{array}{cc}
        I & C^{-1} \\
        C & -I \\
      \end{array}
    \right),
$$
we have
\begin{enumerate}
\item $JH=(JH)^*$,
\item $0\in\sigma_{app}(H)$ or, equivalently, $H$ is not bounded from below.
\end{enumerate}
In fact, the first assertion obviously holds since the diagonal entries of $H$ are bounded.
To prove the second assertion, we see from $\mathcal{R}(I|_{\mathcal{D}(C)}\  \  C^{-1})\subset\mathcal{D}(C)$
that $\mathcal{R}(H)\neq X\times X$, and so by Theorem \ref{th3.1} we conclude that $H$ is not bounded from below.
\end{example}
The following is an application to theory of elasticity.
For the definitions and properties of differential operators not given here,
see \cite{Nai1968, Weid1987}.
\begin{example}
Consider the rectangular thin plate bending problem with two opposite edges simply supported.
The governing equation in terms of displacement is
\begin{align}\label{eq3.4}
D(\frac{\partial^2}{\partial x^2}+\frac{\partial^2}{\partial y^2})^2w=0,\mbox{~for $0<x<h$ and $0<y<1$},
\end{align}
where $D>0$ is a constant,
the boundary conditions for simply supported edges are
\begin{align}\label{eq3.5}
w=0,\ \frac{\partial^2w}{\partial y^2}=0, \mbox{~for $y=0$ or $y=1$},
\end{align}
and the boundary conditions for the other two edges are
\begin{align}\label{eq3.6}
w,\frac{\partial w}{\partial y}=\mbox{given functions, for $x = 0$ or $x = h$},
\end{align}
see \cite[Section 8.1]{YZL}.
To obtain the analytical solution of the above boundary value problem,
the key step is rewriting \eqref{eq3.4} and \eqref{eq3.5} into an operator equation, see \cite{Zhong,ZZ}.
Specifically, we introduce the rotation $\theta$, the Lagrange parametric function $q$, and the moment $m$ as follows \cite{ZZ}.
\begin{align*}
\theta:=\frac{\partial w}{\partial x},\
q:=D(\frac{\partial^3w}{\partial x^3}+\frac{\partial^3w}{\partial x\partial y^2}),\
m:=-D(\frac{\partial\theta}{\partial x}+\frac{\partial^2w}{\partial y^2}).
\end{align*}
Then \eqref{eq3.4} and \eqref{eq3.5} becomes \cite{ZZ}
\begin{align}\label{eq3.7}
&\frac{\partial}{\partial x}\left(
                             \begin{array}{c}
                               w \\
                               \theta \\
                               q \\
                               m \\
                             \end{array}
                           \right)=
                           \left(
                             \begin{array}{cccc}
                               0 & 1 & 0 & 0 \\
                               -\frac{\partial^2}{\partial y^2} & 0 & 0 & -\frac{1}{D} \\
                               0 & 0 & 0 & \frac{\partial^2}{\partial y^2} \\
                               0 & 0 & -1 & 0 \\
                             \end{array}
                           \right)
                           \left(
                             \begin{array}{c}
                               w \\
                               \theta \\
                               q \\
                               m \\
                             \end{array}
                           \right),\\
\nonumber &w=m=0 \mbox{~for $y=0$ or $y=1$}.
\end{align}
Next we write \eqref{eq3.7} as an operator equation in a Hilbert space.
Let
\begin{align*}
\mathcal{D}(T):=&\{f\in L^2(0,\ 1)~|~f, f'\in AC[0,\ 1], \ f''\in L^2(0,\ 1),\ f(0)=f(1)=0\}, \\
Tf:=&-f'' \mbox{~for $f\in\mathcal{D}(T)$}.
\end{align*}
Then $T$ is a self-adjoint linear operator on the Hilbert space $L^2(0,\ 1)$ such that $(T-\pi^2)$ is nonnegative.
Let
$$A:=\left(
     \begin{array}{cc}
       0 & I \\
       T & 0 \\
     \end{array}
   \right), B:=\left(
                \begin{array}{cc}
                  0 & 0 \\
                  0 & -\frac{1}{D}I \\
                \end{array}
              \right).$$
Then \eqref{eq3.7} becomes $\dot{u}=Hu$, where
$$H:=\left(
     \begin{array}{cc}
       A & B \\
       0 & -A^* \\
     \end{array}
   \right)$$
and $u:=(w\  \theta \ q\  m)^t$, so that the spectral properties of the operator $H$ are essential for us to
get the analytical solution of the boundary value problem \eqref{eq3.4}, \eqref{eq3.5}, and \eqref{eq3.6}.
We claim that $i\mathbb R\subset\rho(H)$.
It is sufficient to show $i\mathbb R\subset\rho(-H)$.
Obviously, $H$ is a $J$-self-adjoint Hamiltonian operator matrix on the Hilbert space $(L^2(0,\ 1))^2\times(L^2(0,\ 1))^2$.
Note that $-H$ is a nonnegative Hamiltonian operator matrix,
by Corollary \ref{cor3.1} it is sufficient to show $i\mathbb R\subset\rho(-A)$ or, equivalently, $i\mathbb R\subset\rho(A)$.
Let $\lambda\in\mathbb R$. It follows from
$$A^2=\left(
        \begin{array}{cc}
          T & 0 \\
          0 & T \\
        \end{array}
      \right)$$
that $-\lambda^2\in\rho(A^2)$, so that
$(A-i\lambda)$ is a bijective since
$$A^2+\lambda^2=(A-i\lambda)(A+i\lambda)=(A+i\lambda)(A-i\lambda).$$
Hence $i\lambda\in\rho(A)$.
\end{example}

\section{proof of Theorem \ref{th3.1}}\label{sec4}

We define Pauli operator matrices \cite[p. 3]{Tha} on $X\times X$ as follows.
$$\sigma_0:=\left(
        \begin{array}{cc}
          I & 0 \\
          0 & I \\
        \end{array}
      \right), \sigma_1:=\left(
        \begin{array}{cc}
          0 & I \\
          I & 0 \\
        \end{array}
      \right), \sigma_2:=\left(
                     \begin{array}{cc}
                       0 & -iI \\
                       iI & 0 \\
                     \end{array}
                   \right), \sigma_3:=\left(
                                  \begin{array}{cc}
                                    I & 0 \\
                                    0 & -I \\
                                  \end{array}
                                \right).
$$
The Pauli operator matrices have the following properties.
\begin{align*}
&J=i\sigma_2,\\
&\sigma_k^*=\sigma_k^{-1}=\sigma_k, k=0,1,2,3,\\
&\sigma_1\sigma_2=i\sigma_3=-\sigma_2\sigma_1, \sigma_2\sigma_3=i\sigma_1=-\sigma_3\sigma_2, \sigma_3\sigma_1=i\sigma_2=-\sigma_1\sigma_3.
\end{align*}
Throughout of this section,
$$H:=\left(
       \begin{array}{cc}
         A & B \\
         C & -A^* \\
       \end{array}
     \right)$$
will denote a Hamiltonian operator matrix.

\begin{proposition}\label{prop4.1}
$H$ is nonnegative if and only if
$\sigma_k(i\varepsilon_{1k}H)\sigma_k^*$ is $\sigma_1$-dissipative for some $k=0,1,2,3$, where
\begin{align*}
\varepsilon_{1k}:=\left\{
                  \begin{array}{cc}
                      1, & \mbox{if $\sigma_1\sigma_k=\sigma_k\sigma_1$,} \\
                     -1, & \mbox{if $\sigma_1\sigma_k=-\sigma_k\sigma_1$.}
                  \end{array}
                  \right.
\end{align*}
\end{proposition}
\begin{proof}
The assertions are immediate from
\begin{align*}
{\rm Im}(\sigma_1(iH)x,x)&=(Cx_1,x_1)+(Bx_2,x_2)\\
                         &={\rm Im}(\sigma_1\sigma_3(-iH)\sigma_3^*x,x) \mbox{~for~} x=(x_1\ x_2)^t\in\mathcal{D}(H),\\
{\rm Im}(\sigma_1\sigma_1(iH)\sigma_1^*x,x)&=(Cx_2,x_2)+(Bx_1,x_1)\\
                                          &={\rm Im}(\sigma_1\sigma_2(-iH)\sigma_2^*x,x) \mbox{~for~} x=(x_2\ x_1)^t\in\mathcal{D}(H).
\end{align*}
\end{proof}
\begin{remark}
It seems that the case $k=0$ of Proposition \ref{prop4.1} was first proved by Denisov \cite[Theorem 2]{Deni},
see also \cite[Lemma 4.8]{TW}.
\end{remark}

\begin{proposition}\label{prop4.2}
$\sigma_k(iH)\sigma_k^*$ is $\sigma_2$-symmetric for $k=0,1,2,3$.
\end{proposition}
\begin{proof}
One readily checks that
$(\sigma_2(iH)x,x)=(x,\sigma_2(iH)x)$ for all $x\in\mathcal{D}(H)$,
this proved the case $k=0$ and other cases follow from this.
\end{proof}
\begin{remark}
The case $k=0$ of Proposition \ref{prop4.2} is well-known.
\end{remark}

\begin{lemma}\label{lem4.1}
The following assertions hold.
\begin{enumerate}
\item Every $\mathcal{J}$-dissipative operator can be extended into a maximal one.
\item If $T$ is a maximal $\mathcal{J}$-dissipative operator with a dense domain, then $T$ is closed.
\item A densely defined operator $T$ is maximal $\mathcal{J}$-dissipative if and only if
$-T^c$ is maximal $\mathcal{J}$-dissipative, where $T^c=\mathcal{J}T^*\mathcal{J}$ is the $\mathcal{J}$-adjoint of the operator $T$.
\end{enumerate}
\end{lemma}
\begin{proof}
See \cite[Section 2.2]{Azi}.
\end{proof}

\begin{proposition}\label{prop4.3}
Let $H$ be a nonnegative Hamiltonian operator matrix.
Then $iH$ is maximal $\sigma_1$-dissipative if and only if $iH$ is $\sigma_2$-self-adjoint or, equivalently, $H$ is $J$-self-adjoint.
\end{proposition}
\begin{proof}
Assume $iH$ is maximal $\sigma_1$-dissipative.
Then $-(iH)^c$ is $\sigma_1$-dissipative by Lemma \ref{lem4.1},
and so $\sigma_3(iH)^c\sigma_3$ is $\sigma_1$-dissipative by Proposition \ref{prop4.1}.
Noting that $iH\subset\sigma_2(iH)^*\sigma_2$ by Proposition \ref{prop4.2} and that $(iH)^*=\sigma_1(iH)^c\sigma_1$,
we have $iH\subset\sigma_3(iH)^c\sigma_3$, and therefore
$iH=\sigma_3(iH)^c\sigma_3$ since $iH$ is maximal $\sigma_1$-dissipative.
Hence $iH=\sigma_2(iH)^*\sigma_2$.
Conversely, if $iH=\sigma_2(iH)^*\sigma_2$,
then $iH=\sigma_3(iH)^c\sigma_3$,
and so $-(iH)^c=-\sigma_3(iH)\sigma_3$ is $\sigma_1$-dissipative by Proposition \ref{prop4.1}.
Let $G$ be a maximal $\sigma_1$-dissipative extension of $iH$.
We get $-G^c\subset -(iH)^c$ since $iH\subset G$.
But $G^c$ is maximal $\sigma_1$-dissipative by Lemma \ref{lem4.1},
and so $-G^c=-(iH)^c$. Thus $G=iH$ since $G$ is closed by Lemma \ref{lem4.1}.
\end{proof}

\begin{proposition}\label{prop4.4}
Let $H$ be a nonnegative Hamiltonian operator matrix.
Then $H$ is $J$-self-adjoint if and only if $\mathcal{R}(H+\sigma_1)=X\times X$.
\end{proposition}
\begin{proof}
By Lemma 2.2 in \cite[Section 2.2]{Azi},
$iH$ is maximal $\sigma_1$-dissipative if and only if $\mathcal{R}(\sigma_1(iH)+i)=X\times X$ or, equivalently,
$\mathcal{R}(H+\sigma_1)=X\times X$.
Thus, the assertion follows from Proposition \ref{prop4.3}.
\end{proof}

\begin{proposition}\label{prop4.5}
Let $H$ be a nonnegative Hamiltonian operator matrix.
Then
$$\mathcal{N}(H)=(\mathcal{N}(A)\cap\mathcal{N}(C))\times(\mathcal{N}(B)\cap\mathcal{N}(A^*)).$$
\end{proposition}
\begin{proof}
Let $x=(x_1\ x_2)^t\in\mathcal{N}(H)$.
Then
\begin{align}\label{eq2.2}
Ax_1+Bx_2=0, \\
\nonumber Cx_1-A^*x_2=0,
\end{align}
and so
\begin{align*}
(Ax_1,x_2)+(Bx_2,x_2)=0,\\
(x_1,Cx_1)-(x_1,A^*x_2)=0,
\end{align*}
so that $(x_1,Cx_1)+(Bx_2,x_2)=0$.
Thus $Cx_1=0, Bx_2=0$ since both $B$ and $C$ are nonnegative self-adjoint operators.
Consequently $Ax_1=0, A^*x_2=0$ by \eqref{eq2.2}, and therefore
$x\in(\mathcal{N}(A)\cap\mathcal{N}(C))\times(\mathcal{N}(B)\cap\mathcal{N}(A^*))$.
Hence
$$\mathcal{N}(H)\subset(\mathcal{N}(A)\cap\mathcal{N}(C))\times(\mathcal{N}(B)\cap\mathcal{N}(A^*)).$$
The reverse inclusion is obvious.
\end{proof}

\emph{Proof of Theorem \ref{th3.1}.}
If $\mathcal{R}(H)=X\times X$, then $\mathcal{R}(JH)=X\times X$,
and so $JH$ is a self-adjoint operator with $0\in\rho(JH)$.
Thus (1), (2), and (3) are equivalent.
By Proposition \ref{prop4.3}, (3) is equivalent to (5), and (4) is equivalent to (6).
Finally, a closed operator is bounded from below if and only if it is injective and has a closed range
(see Theorem I.3.7 and Lemma IV.1.1 in \cite{Gold}),
and therefore (3) is equivalent to (4) by Proposition \ref{prop4.5}. \qed

{\textbf{Acknowledgment.}}~{\ This work was supported by
the Natural Science Foundation of China (Grant Nos. 11361034, 11371185, 11101200),
the Specialized Research Fund for the Doctoral Program of Higher Education of China (Grant No. 20111501110001),
the Major Subject of Natural Science Foundation of Inner Mongolia of China (Grant No. 2013ZD01),
and the Natural Science Foundation of Inner Mongolia of China (Grant No. 2012MS0105).}


\begin{thebibliography}{99}
\bibitem{AHF}Alatancang, J. Huang, X. Fan,
Structure of the spectrum of infinite dimensional Hamiltonian operators,
Sci. China Ser. A 51 (5) (2008) 915-924.
\bibitem{AJW}Alatancang, G. Jin, D. Wu,
On symplectic self-adjointness of Hamiltonian operator matrices,
arXiv:1305.5910, 2013.
\bibitem{Atk}F. V. Atkinson, Discrete and continuous boundary probelms, Academic Press, New York, 1964.
\bibitem{Azi}T. Y. Azizov, I. S. Iokhvidov, Linear operators in spaces with an indefinite metric,
John Wiley and Sons, Chichester, 1989.
\bibitem{AD1}T. Y. Azizov, A. A. Dijksma, On the boundedness of Hamiltonian operators,
Proc. Amer. Math. Soc. 131 (2) (2003) 563-576.
\bibitem{AKK}T. Y. Azizov, V. K. Kiriakidi, G. A. Kurina,
An indefinite approach to the reduction of a nonnegative Hamiltonian operator function to a block diagonal form,
Funct. Anal. Appl. 35 (3) (2001) 220-221.
\bibitem{CM}P. R. Chernoff, J. E. Marsden,
Properties of infinite dimensional Hamiltonian systems (Lecture Notes in Mathematics 425), Springer-Verlag, Berlin, 1974.
\bibitem{Deni}M. S. Denisov,
Invertibility of linear operator in the Krein space (in Russian),
Gos. Univ. Math. Phys. (2) (2005) 133-137.
\bibitem{GK}I. C. Gohberg, M. G. Krein, Introduction to the theory of linear nonselfadjoint operators,
American Mathematical Society, Providence, 1969.
\bibitem{Gold}S. Goldberg, Unbounded linear operators: theory and applications, McGraw-Hill, New York, 1966.
\bibitem{Hal}P. R. Halmos, A Hilbert space problem book, second edition, Springer-Verlag, New York, 1982.
\bibitem{HK}M. H\v{a}r\v{a}gu\c{s}, T. Kapitula, On the spectra of periodic waves for infinite-dimensional Hamiltonian systems,
Phys. D 237 (20) (2008) 2649-2671.
\bibitem{HK1970}P. Hess, T. Kato, Perturbation of closed operators and their adjoints,
Comment. Math. Helv. 45 (1970) 524-529.
\bibitem{Ka1980}T. Kato, Perturbation theory for linear operators, Corrected Printing of the Second Edition,
Springer-Verlag, Berlin, 1980.
\bibitem{Krall2002}A. M. Krall, Hilbert spaces, boundary value problems and orthogonal polynomials, Birkh\"{a}user Verlag, Basel, 2002.
\bibitem{KZ}C. R. Kuiper, H. J. Zwart,
Connections between the algebraic Riccati equation and the Hamiltonian for Riesz-spectral systems,
J. Math. Systems Estim. Control 6 (4) (1996) 1-48.
\bibitem{Ku1}G. A. Kurina, Invertibility of nonnegatively Hamiltonian operators in a Hilbert space,
Differential Equations 37 (6) (2001) 880-882.
\bibitem{KM}G. A. Kurina, G. V. Martynenko,
On the reducibility of a nonnegatively Hamiltonian periodic operator function in a real Hilbert space to a block diagonal form,
Differential Equations 37 (2) (2001) 227-233.
\bibitem{LRW}H. Langer, A. C. M. Ran, B. A. van de Rotten,
Invariant subspaces of infinite dimensional Hamiltonians and solutions of the corresponding Riccati equations,
Linear operators and matrices, 235-254,
Oper. Theory Adv. Appl., 130, Birkh\"{a}user, Basel, 2002.
\bibitem{LX}C. W. Lim, X. S. Xu, Symplectic elasticity: theory and applications,
Applied Mechanics Reviews, 63 (5) (2010), article ID 050802.
\bibitem{Nai1968}M. A. Naimark, Linear differential operators, Part II: Linear differential operators in Hilbert space,
Frederick Ungar, New York, 1968.
\bibitem{Olver}P. J. Olver, Applications of Lie groups to differential equations (Graduate Texts in Mathematics 107),
second edition, Springer-Verlag, New York, 1993.
\bibitem{QC1}J. Qi, S. Chen, Essential spectra of singular matrix differential operators of mixed order in the limit circle case,
Math. Nachr. 284 (2-3) (2011) 342-354.
\bibitem{QC2}J. Qi, S. Chen, Essential spectra of singular matrix differential operators of mixed order,
J. Differential Equations 250 (12) (2011) 4219-4235.
\bibitem{SS}H. Sun, Y. Shi, Self-adjoint extensions for linear Hamiltonian systems with two singular endpoints,
J. Funct. Anal.  259 (8) (2010) 2003-2027.
\bibitem{Tha}B. Thaller, The Dirac equation (Theoretical and Mathematical Physics), Springer-Verlag, Berlin, 1992.
\bibitem{TW}C. Tretter, C. Wyss,
Dichotomous Hamiltonians with Unbounded Entries and Solutions of Riccati Equations,
arXiv:1304.5921, 2013.
\bibitem{Weid1987}J. Weidmann, Spectral Theory of Ordinary Differential Operators (Lecture Notes in Mathematics 1258),
Springer-Verlag, Berlin, 1987.
\bibitem{WA}D. Wu, Alatancang,
Invertibility of nonnegative Hamiltonian operator with unbounded entries
J. Math. Anal. Appl.  373 (2) (2011) 410-413.
\bibitem{WA2}D. Wu, Alatancang, Symplectic self-adjointness of infinite dimensional Hamiltonian operators (in Chinese),
Acta Math. Appl. Sin. 34 (5) (2011) 918-923.
\bibitem{Wyss2011}C. Wyss,
Hamiltonians with Riesz bases of generalised eigenvectors and Riccati equations,
Indiana Univ. Math. J., 60 (5) (2011) 1723-1766.
\bibitem{YZL}W. Yao, W. Zhong, C. W. Lim,
Symplectic elasticity, World Scientific, New Jersey, 2009.
\bibitem{ZAZ}H. Zhang, Alatancang, W. Zhong,
The Hamiltonian system and completeness of symplectic orthogonal system,
Appl. Math. Mech. (English Ed.) 18 (3) (1997)237-242.
\bibitem{Zhong}W. Zhong, A new systematic method in elasticity theory (in Chinese),
Dalian University of Technology Press, Dalian, 1995.
\bibitem{ZZ}W. Zhong, X. Zhong,
Method of separation of variables and Hamiltonian system,
Numer. Methods Partial Differential Equations 9 (1) (1993) 63-75.
\end{thebibliography}
\end{document}